\documentclass[12pt,oneside]{amsart}
\usepackage{amsmath}
\usepackage{amsfonts}
\usepackage{amsthm}
\usepackage{amssymb}
\usepackage[all]{xy}
 \SelectTips{cm}{}

\title{A case of monoidal uniqueness of algebraic models}

\author{Constanze Roitzheim}
\email{C.Roitzheim@kent.ac.uk}

\address{C. Roitzheim \\ University of Kent \\ School of Mathematics, Statistics and Actuarial Science \\ Canterbury \\ Kent CT2 7NF,
UK}


\date{$10^\text{th}$ July 2013}

\DeclareMathOperator{\Ho}{Ho}
\DeclareMathOperator{\End}{End}

\DeclareMathOperator{\Sp}{Sp}

\newcommand{\Zp}{\mathbb{Z}_p}
\newcommand{\C}{\mathcal{C}}
\newcommand{\D}{\mathcal{D}}

\newcommand{\LS}{{L_{K(1)}\mathbb{S}}}

\newtheorem{theorem}{Theorem}[section]
\newtheorem{proposition}[theorem]{Proposition}
\newtheorem{corollary}[theorem]{Corollary}
\newtheorem{lemma}[theorem]{Lemma}

\newtheorem*{theorem*}{Theorem}

\newtheorem{definition}[theorem]{Definition}

\newtheorem*{thm}{Theorem}

\newcommand{\lradjunction}{\,\,\raisebox{-0.1\height}{$\overrightarrow{\longleftarrow}$}\,\,}

\begin{document}
\begin{abstract}
We prove that there is at most one algebraic model for modules over the $K(1)$-local sphere at odd primes that retains some monoidal information.
\end{abstract}
\maketitle

\section*{Introduction}

For the last few decades, the stable homotopy category $\Ho(\Sp)$ has been studied using chromatic localisations. The $K(n)$-local stable homotopy categories \linebreak $\Ho(L_{K(n)}\Sp)$ and $E(n)$-local stable homotopy categories $\Ho(L_{E(n)}\Sp)$ contain a wealth of information on the finer structure of the stable homotopy category itself. More recently, the question has been asked where and how the higher homotopy information of spectra or $E$-local spectra is encoded. For example, Schwede showed in \cite{Sch07} that all higher homotopy information of spectra lies in the triangulated structure of $\Ho(\Sp)$ itself, i.e. the stable homotopy category is \emph{rigid}. But what about rigidity for chromatic localisations? 

The author showed that the $K$-local stable homotopy category at $p=2$ is rigid, providing the first example of chromatic rigidity. However, at odd primes the situation is already very different as Franke provided a potential counterexample in \cite{Fra96}. He constructed an algebraic model category whose homotopy category is equivalent to $\Ho(L_{K_{(p)}}\Sp)$ but which has very different homotopical behaviour. And indeed, it is still mysterious in general whether a piece of the stable homotopy category can have an algebraic model or whether it might in fact be rigid.

Since then it has been a subject of investigation to see if certain homotopy categories of modules over a fixed ring spectrum are rigid or possess an exotic model. In \cite{Hut12} Hutschenreuter showed a version of rigidity for $ko$-module spectra. Patchkoria modified Franke's construction to obtain algebraic models for modules over the truncated Brown-Petersen spectrum for certain $n$ and $p$, modules over $E(n)$ for certain $n$ and $p$, $KU$, $k(n)$ and $ko_{(p)}$ for $p$ odd \cite{Pat12}. 

Another interesting question arising from this is the uniqueness of exotic models, specifically algebraic ones. Can there be more than one exotic model for a given homotopy category? Are all such exotic models algebraic? The result of this paper adds to the known landscape of rigidity vs. algebraic models. We study the case of modules over the $K(1)$-local sphere for odd primes and show that there can be at most one algebraic model retaining some monoidal information. 

\begin{thm}
Let $$\Phi: \Ho((\LS)\mbox{-mod}) \longrightarrow \Ho(\C)$$ be a pre-monoidal equivalence with $\C$ an algebraic model category. Then $\C$ is unique: any other algebraic model category $\D$ with a pre-monoidal triangulated equivalence $$\Psi: \Ho((\LS)\mbox{-mod}) \longrightarrow \Ho(\D)$$ is Quillen equivalent to $\C$. 
\end{thm}

We do not know if $\Ho((L_{K(1)}\mathbb{S})\mbox{-mod})$ possesses an algebraic model at all, but if it does, then it is unique. We hope that the methods and results that we present will provide a stepping stone towards other uniqueness results in the future.

\bigskip
This paper is organised as follows. In Section \ref{sec:background} we recall some basic techniques including Morita theory for stable model categories and algebraic model categories. This reduces the study of algebraic models to the study of endomorphism differential graded algebras (=dgas). In Section \ref{sec:rigidity} we describe the set-up for this paper, particularly the algebraic structure we would like to capture. We explain how it is our goal to show that there is only one commutative dga whose homology groups and Massey products agree with the homotopy groups and Toda brackets of the $K(1)$-local sphere. Section \ref{sec:postnikov} contains a helpful technical result that we will use in our main computation. This main computation is contained in Section \ref{sec:main}: we show that there is only one dga that has indeed the right algebraic information. We finish with some conclusions and further directions in Section \ref{sec:conclusions}.

\bigskip
{\bf Acknowledgments:} The author would like to thank David Barnes and Brooke Shipley for helpful and motivating discussions.

\section{Stable model categories and algebraic model categories }\label{sec:background}

In this section we are going to re-introduce some model category techniques, especially focussing on stable model categories, Morita theory for stable model categories, algebraic model categories and Eilenberg-MacLane spectra. 

\subsection{Background}

We assume that the reader is familiar with the basic constructions concerning model categories such as the homotopy category, Quillen functors and Quillen equivalences. Excellent references are provided by \cite{DwySpa95} and \cite{Hov99}. 

Recall that a pointed model category $\C$ is \emph{stable} if the adjoint loop and suspension functors
\[
\Sigma: \Ho(\C) \lradjunction \Ho(\C): \Omega
\]
are equivalences of categories. For example, $\C= Ch(\mathcal{A})$ (unbounded chain complexes in an abelian category $\mathcal{A}$) is a stable model category whereas the category of topological spaces is not. The most important feature of a stable model category is that its homotopy category is triangulated. Furthermore, Quillen functors between stable model categories induce functors of the respective homotopy categories that are compatible with the triangulated structure.

Throughout the paper we also assume that all our model categories are proper and cellular. This is only a minor technical restriction, see e.g. \cite[Section 8]{BarRoi13} for a justification.

Another prime example of a stable model category is the model category of symmetric spectra, $\Sp$. In this paper we also consider \emph{Bousfield localisations} of spectra with respect to a generalised homology theory $E$, see e.g. \cite[Section 1]{BarRoi11b}. The category of spectra with the $E$-local model structure is denoted $L_E\Sp$. Its resulting homotopy category, the \emph{$E$-local stable homotopy category}, is obtained from the stable homotopy category $\Ho(\Sp)$ by formally inverting the $E_*$-isomorphisms of spectra. 

Hence, the $E$-local stable homotopy category is especially sensitive to phenomena related to the homology theory $E$. This becomes useful when considering certain homology theories employed to study the internal structure of the stable homotopy category itself, such as the Morava-$K$-theories $K(n)$ and the Johnson-Wilson homology theories $E(n)$. Those homology theories are related to nilpotency and periodicity phenomena  as well as other structural results such as the thick subcategory theorem (which says that the $K(n)$ essentially provide the ``atomic'' localisations of $\Ho(\Sp)$) and the Chromatic Convergence Theorem (which states that for bigger and bigger $n$, the $E(n)$-local stable homotopy category provides a better and better approximation of $\Ho(\Sp)$ itself), see e.g. \cite{HopSm98} \cite{Rav92}. 

\subsection{Morita theory}

Morita theory for stable model categories has proved to be an extremely powerful and valuable way to study and classify stable model categories with a single compact generator. A \emph{generator} of a stable model category is an object detecting isomorphisms in $\Ho(\C)$. An object $X$ of $\Ho(\C)$ is called \emph{compact} if the functor $[X, -]^\C$ commutes with arbitrary coproducts. (Here, $[-,-]^\C$ denotes the set of morphisms in $\Ho(\C)$.) Without loss of generality we assume our generators to be fibrant and cofibrant. If a model category has a compact generator $X$, it means that the entire homotopy category can be ``generated'' from $X$ under coproducts and exact triangles \cite[Section 4.2]{Kel94}. 

In \cite{SchShi03} Schwede and Shipley construct an \emph{endomorphism ring spectrum} $\End(X)$ of a compact generator $X$ and show the following. 

\begin{theorem}[Schwede-Shipley]\label{thm:morita}
Let $\C$ be a stable model category with a compact generator $X$. Then $\C$ is Quillen equivalent to the model category of modules over the ring spectrum $\End(X)$. 
\end{theorem}

This shows that all homotopy information of a monogenic stable model category is encoded in a symmetric ring spectrum. A version of the above theorem exists for model categories with several compact generators, but we only need to be interested in the monogenic case for this paper.

\subsection{Algebraic model categories and Eilenberg-MacLane spectra}\label{subsec:algebraic}

We would now like to focus on algebraic model categories. 

\begin{definition}
A model category is called \emph{algebraic} if it is enriched, tensored and cotensored over the model category of chain complexes of abelian groups $Ch(\mathbb{Z})$. Further, this has to be a Quillen adjunction in two variables in the sense of \cite[Definition 4.2.1]{Hov99}.
\end{definition}

For details see also \cite[Appendix A]{Dug06}. Applying the Morita theorem \ref{thm:morita} to an algebraic model category $\C$ with a single compact generator $X$ yields something special:

\begin{theorem}[Dugger-Shipley]\label{thm:algebraicmorita}
Let $\C$ be an algebraic model category with a single cofibrant and fibrant compact generator $X$. Then the endomorphism ring spectrum of $X$ in the sense of Theorem \ref{thm:morita} is the Eilenberg-MacLane spectrum $EML(\C(X,X))$ of the endomorphism dga $\C(X,X)$. 
\end{theorem}

This is \cite[Proposition 6.3]{DugShi07}.
Details on the definition of Eilenberg-MacLane spectra of dgas can be found in \cite[Section 1.2]{DugShi07b}. To give a very brief summary, Shipley showed in \cite{Shi07} that the model category of chain complexes is Quillen equivalent to the model category of symmetric spectra in simplicial abelian groups. The monoidal version of this states that dgas are Quillen equivalent to symmetric ring spectra in simplicial abelian groups. The Eilenberg-MacLane functor from symmetric ring spectra in simplicial abelian groups to symmetric ring spectra (in simplicial sets) is simply induced by the forgetful functor from abelian groups to sets. 

The underlying Eilenberg-MacLane spectrum (ignoring the ring structure) of a dga $C$ is entirely determined by the homology $H_*(C)$. However, the multiplication on $EML(C)$ depends on more than the homology alone. Eilenberg-MacLane spectra of dgas satisfy the following property: if $C$ and $D$ are quasi-isomorphic dgas, then their Eilenberg-MacLane spectra are weakly equivalent as ring spectra. This statement is not an equivalence as Dugger and Shipley show that there are dgas that give rise to weakly equivalent EML-ring spectra without being quasi-isomorphic \cite[Section 5.1]{DugShi07}.

\section{Rigidity and algebraic models}\label{sec:rigidity}

We would like to see whether certain Bousfield localisations of the stable homotopy category $\Ho(\Sp)$ possess algebraic models and if so, how many. Specifically, our goal is to show that there is at most one algebraic model realising the stable model category $(\LS)\mbox{-mod}$ of modules over the $K(1)$-local sphere spectrum. Throughout this paper we only consider localisations at an odd prime $p$.

\subsection{Massey products and Toda brackets}\label{subsec:toda}

Assume that there is an equivalence of triangulated categories
\[
\Phi: \Ho((\LS)\mbox{-mod}) \longrightarrow \Ho(\C)
\]
where $\C$ is an algebraic model category. The left hand side has one compact generator, namely $\LS$ itself. Thus, a (cofibrant and fibrant replacement of) $X:= \Phi(\LS)$ is a compact generator of $\C$. By the Morita theory results outlined in Section \ref{subsec:algebraic}, $\C$ is Quillen equivalent to modules over the Eilenberg-MacLane ring spectrum of the dga $\C(X,X)$. In order to study this Eilenberg-MacLane spectrum we first make some observations about $\C(X,X)$ itself.

For that, we briefly recall the definition of Massey products which can also be found in e.g. \cite[Chapter A1.4]{Rav86} and \cite[Remark 5.10]{BenKraSch05}. (Note that in the definition given in those references, dgas are graded cohomologically whereas in this paper we grade dgas homologically.) Let $C$ be a dga, and let $\alpha$, $\beta$ and $\gamma$ be elements in the homology $H_*(C)$ such that $\alpha \cdot \beta =0$ and $\beta \cdot \gamma =0$. Say that $\alpha$, $\beta$ and $\gamma$ are represented by cycles $a$, $b$ and $c$ respectively. Because of our assumption there is at least one element $u \in C$ such that $$d(u)=(-1)^{1+|a|}a \cdot b$$ and at least one element $v$ such that $$d(v)=(-1)^{1+|b|}b\cdot c.$$

The \emph{triple Massey product} of $\alpha$, $\beta$ and $\gamma$ is the coset
\begin{multline}
\left< \alpha, \beta, \gamma \right> = \{ [(-1)^{1+|u|}u \cdot c + (-1)^{1+|a|}a \cdot v] \,\,\,|\,\,\, d(u)=(-1)^{1+|a|}a \cdot b, d(v)=(-1)^{1+|b|}b\cdot c \} \nonumber \\ \subseteq H_{|a|+|b|+|c|+1}(C) \nonumber
\end{multline}

Because those elements $u$ and $v$ might not be unique, the Massey product can consist of more than one element. The choices are encoded in the group
\[
\alpha H_{|b|+|c|+1}(C) \oplus \gamma H_{|a|+|b|+1}(C)
\]
which is called the \emph{indeterminacy} of $\left<\alpha, \beta,\gamma\right>.$ If the indeterminacy is trivial, the Massey product consists of one element only which will be the case for our examples in this paper.

\begin{lemma}\label{lem:homology}
The homology of the endomorphism dga $\C(X,X)$ is isomorphic, as a graded algebra, to the homotopy groups $\pi_*(\LS)$. Moreover, under this isomorphism the Massey products of $\C(X,X)$ agree with the Toda brackets of $\pi_*(\LS)$. 
\end{lemma}

\begin{proof}
The isomorphism in question is given by
\[
H_*(\C(X,X)) \cong [X,X]^\C_* \cong [\LS, \LS]^{\LS\mbox{\scriptsize{-mod}}}_* = \pi_*(\LS).
\]
The first isomorphism is given by the following adjunction, see e.g \cite[Appendix A]{Dug06},
\[
[X,X]^\C_* \cong [X \otimes^L \mathbb{Z}[0] ,X]^\C_* \cong \Ho(Ch(\mathbb{Z}))(\mathbb{Z}[0], \C(X,X)) = H_*(\C(X,X))
\]
where $\otimes^L$ is part of the $Ch(\mathbb{Z})$-model structure.
The second isomorphism is given by the equivalence $\Phi.$ The relation between Massey products and Toda brackets is shown in e.g. \cite[Sections 3-5]{Ada66}.

\end{proof}

Let us make this homotopy and Toda bracket structure concrete: The homotopy groups of the $K(1)$-local sphere at an odd prime $p$ are given by the following.

\[
\pi_i L_{K(1)}\mathbb{S} =  \left\{ \begin{array}{r@{\quad:\quad}l} \Zp & i=0,-1 \\ \mathbb{Z}/ p^{\nu(m)+1} & i=(2p-2)m-1, m \neq 0 \\ 0 & \mbox{else.}   \end{array} \right.
\]

Here, $\mathbb{Z}_p$ denotes the $p$-adic integers and $\nu(m)$ denotes how often $p$ occurs in the prime factor decomposition of the integer $m$. 
Let $\alpha_i$ be an order $p$ element of $\pi_{(2p-2)i-1}(L_{K(1)}\mathbb{S})$, $i \neq 0$. Then the only nontrivial Toda brackets of length 3 are
\[
\alpha_{i+j}=\left< \alpha_i,p,\alpha_j \right>.
\]
Their indeterminacy is zero. (For more detailed information, see e.g. \cite{Rav86}.)

The algebra structure and Toda bracket structure on $\pi_*(\LS)$ and hence the corresponding algebra and Massey product structure on $H_*(\C(X,X))$ are both very strong. For degree reasons we see immediately that there is no nontrivial multiplication on $\pi_*(\LS)$. The given homotopy groups are all cyclic of $p$-power order (except in degrees $0$ and $-1$). While the additive generators of the respective homotopy groups are not always captured by Toda brackets, their order-$p$ elements of are entwined in very rigid Toda bracket relations.

\bigskip
Now assume that there are two algebraic model categories, $\C$ and $\D$, whose homotopy categories are triangulated equivalent to $\Ho((\LS)\mbox{-mod})$. Let us denote the compact generators of $\C$ and $\D$ given by the image of $\LS$ under those triangulated equivalences by $X$ and $Y$ respectively. By Theorem \ref{thm:algebraicmorita}, $\C$ is Quillen equivalent to the category of modules over the Eilenberg-MacLane spectrum of the endomorphism dga $\C(X,X)$. Analogously, $\D$ is Quillen equivalent to modules over the Eilenberg-MacLane spectrum of the endomorphism dga $\D(Y,Y)$.

By Lemma \ref{lem:homology}, those two endomorphism dgas have the same homology algebra and Massey product structure (given by the homotopy groups and Toda brackets of $\LS$). As described earlier, this algebra and Massey product structure is very strong- in fact, strong enough to show that  if $\C(X,X)$ and $\D(Y,Y)$ are commutative dgas, then they are quasi-isomorphic. (This is going to be Theorem \ref{th:quasiiso}.) Thus, with the additional assumption of commutativity, their Eilenberg-MacLane ring spectra are weakly equivalent ring spectra and their respective categories of module spectra are Quillen equivalent \cite[Theorem 4.3]{SchShi00}, which is the main result of this paper.

We would also like to point out that all dgas in this paper are dgas over $\Zp$, the $p$-adic integers. 

\subsection{A note on commutativity}\label{sec:commutative}

Unfortunately, at this stage we can only show the main computation for commutative dgas rather than general associative dgas. While the homology of the relevant dgas is commutative, there is no guarantee that every dga realising it is actually commutative itself or even quasi-isomorphic to a commutative dga. 

The endomorphism dga of an object in an algebraic model category will generally not be commutative. However, if the algebraic model category $\C$ in question was a \emph{monoidal} model category with $X$ being the unit, then $\C(X,X)$ would indeed be a commutative dga. For our set-up this can be guaranteed if we use the following additional assumption.

\begin{definition}
Let $$\Phi: \mathcal{M} \longrightarrow \mathcal{N}$$ be a functor of monoidal categories. We say that $\Phi$ is \emph{pre-monoidal} if there is a natural isomorphism 
\[
\Phi( A \otimes B) \cong \Phi(A) \otimes \Phi(B)
\]
for all $A, B \in \mathcal{M}$.
\end{definition}

This is much weaker than assuming $\Phi$ to be a strong monoidal functor as we do not ask for any associativity relations. For example, the important case of Franke's algebraic model and its monoidal behaviour is described in \cite{Gan07} and \cite{BarRoi11a}. A pre-monoidal functor must send the unit of $\mathcal{M}$ to the unit of $\mathcal{N}$. Thus, if we ask our algebraic models to be related to $\Ho((\LS)\mbox{-mod})$ by a pre-monoidal triangulated equivalence rather than just a triangulated equivalence, $\Phi$ must send the unit to the unit. Thus, we are in the situation of our endomorphism dgas being commutative as it is the endomorphism dga of the unit. The key to this is Theorem \ref{th:quasiiso} which we will prove in Section \ref{sec:main}. It states the following.

\begin{theorem*}
Let $D$ be a commutative dga whose homology and Massey products agree with the homotopy groups and Toda brackets of $\pi_*(\LS).$
Then there is a quasi-isomorphism of dgas $$\varphi: C \longrightarrow D$$ where $C$ is given by
\begin{eqnarray}
C = \Zp[x,x^{-1}] \otimes \Lambda_{\Zp}(e) \nonumber \\ \,\,\,\,\, |x|=2p-2, |e|=2p-3,\,\,\,\,\, d(x)=pe. \nonumber
\end{eqnarray}
\end{theorem*}

\medskip
This will then prove the result of this paper:

\begin{theorem}\label{thm:premonoid}
Let $$\Phi: \Ho((\LS)\mbox{-mod}) \longrightarrow \Ho(\C)$$ be a pre-monoidal equivalence with $\C$ an algebraic model category. Then $\C$ is unique: any other algebraic model category $\D$ with a pre-monoidal triangulated equivalence $$\Psi: \Ho((\LS)\mbox{-mod}) \longrightarrow \Ho(\D)$$ is Quillen equivalent to $\C$. 
\end{theorem}

\section{A Postnikov section argument}\label{sec:postnikov}

So we have established that the final goal of this paper is to show that there is only one commutative dga $D$ whose homology and Massey products agree with the homotopy groups and Toda brackets of the $K(1)$-local sphere $\LS$. Before we get to the main computation, we dedicate this section to showing that we can assume $D$ to satisfy
\[
D_0=H_0(D)=\Zp.
\]
This uses Postnikov approximation as outlined in \cite[Section 5]{Shi02}. Postnikov approximations are certain dgas fitting into factorisations produced by the small object argument \cite[Theorem 2.1.14]{Hov99}. For certain sets $S$ and a morphism in a category with enough colimits, the small object argument produces a factorisation of said morphism where the first map is produced out of pushouts and transfinite compositions of $S$, and the second map in the factorisation has the right lifting property with respect to $S$.

Let us turn to our case.
A \emph{Postnikov approximation} of $D$ is a dga $P_0D$ given by a factorisation
\[
D \xrightarrow{i} P_0 D \longrightarrow 0
\]
produced by the small object argument with respect to the set
\[
I = \{ F(S^n) \longrightarrow F(D^{n+1}) \,\,\,|\,\,\, n > 1\}.
\]
Here, $F$ denotes the free commutative dga functor (left adjoint to the forgetful functor from commutative dgas to chain complexes). The chain complex $S^n$ is $\Zp$ in degree $n$ and zero otherwise, and $D^{n+1}$ is the acyclic chain complex over $\Zp$ given by an element $x$ in degree $n$, an element $y$ in degree $n+1$ and the differential $d(y)=x$.

The Postnikov approximation satisfies the following.
\begin{itemize}
\item The map $i$ is an isomorphism in degrees $\leq 1$ and thus induces a homology isomorphism in nonpositive degrees.
\item The map $P_0D \longrightarrow 0$ is a homology isomorphism in positive degrees.
\end{itemize}
The first point follows from the fact that the original dga $D$ has only been ``modified'' in degrees bigger than $1$. The second point follows from $P_0D \longrightarrow 0$ having the right lifting property with respect $I$- by adjunction, this means that $P_0D \longrightarrow 0$ (viewed as a map of chain complexes) has the right lifting property with respect to $S^n \longrightarrow D^{n+1}$ for $n \ge 1$, which implies that it is a homology isomorphism in positive degrees.

Using the small object argument again with respect to the set
\[
J = \{ F(0) \longrightarrow F(D^{n+1}) \,\,\,|\,\,\, n > 1\}
\]
one obtains a factorisation of the map $i: D \longrightarrow P_0D$ as
\[
D \xrightarrow{\bar{i}} \bar{D} \xrightarrow{\bar{p}} P_0 D.
\]
This satisfies
\begin{itemize}
\item The map $\bar{i}$ is a quasi-isomorphism and also an isomorphism in nonpositive degrees.
\item The map $\bar{p}$ is an epimorphism and also an isomorphism in nonpositive degrees.
\end{itemize}

Again, the first point follows from the fact that $D$ has only been ``modified'' in positive degrees, recalling that $\bar{p}\circ\bar{i}=i$ is an isomorphism in nonpositive degrees. The second point is an adjunction argument analogous to the previous case.

We now truncate $P_0D$ to obtain a dga $Q$ given by
\[
Q_i = \left\{ \begin{array}{r@{\quad:\quad}l} 
0 & i >0 \\
H_0(D) & i=0 \\
D_i & i<0.
\end{array} \right.
\]
We can also write down a quasi-isomorphism
\[
h: Q \longrightarrow P_0D.
\]
This map is zero in positive degrees and the identity in negative degrees. In degree zero it is the composition
\[
H_0(D) \longrightarrow ker(d: (P_0 D)_0\rightarrow (P_0D)_{-1}) \subset (P_0D)_0.
\]
The cycle-choosing homomorphism in this composition exists since $H_0(D)=\mathbb{Z}_p$ is projective over $\mathbb{Z}_p$ itself.

\begin{proposition}\label{prop:degreezero}
There is a dga $D'$ which is quasi-isomorphic to $D$ and satisfies
\[
D'\cong H_0(D)=\Zp.
\]
\end{proposition}

\begin{proof}
We consider the pullback square of dgas with the maps $\bar{p}$ and $h$ as before
\[
\xymatrix{ D' \ar[r]^{h'} \ar[d]_{p'} & \bar{D} \ar@{->>}[d]^{\bar{p}} \\
Q \ar[r]_{h}^\sim & P_0D.
}
\]
A pullback square of dgas is a pullback square of the underlying chain complexes.
The model category of chain complexes over $\Zp$ with the projective model structure is right proper. 
This means that a pullback of a weak equivalence along a fibration is again a weak equivalence. 
The vertical map $\bar{p}$ is an epimorphism and hence a fibration of chain compexes. Thus, $h'$ is a homology isomorphism. 

Further, $\bar{p}$ is an isomorphism in nonpositive degrees. Thus, $p'$ is too. Consequently,
\[
D_0'\cong Q_0=H_0(D)=\Zp.
\]
\end{proof}

\section{The main computation}\label{sec:main}

Our goal is to show that there is only one commutative dga over $\Zp$ whose homology and Massey products agree with the homotopy groups and Toda brackets of $\LS$ described in Section \ref{subsec:toda}. We are going to do so as follows: we construct a commutative dga $C$ with the right homology data. Assuming that there is an arbitrary commutative dga $D$ with the same homology data, we can construct a quasi-isomorphism of dgas from $C$ to $D$ from this data alone.

Let us now turn to constructing the test dga $C$.
Let $C$ be the differential graded algebra over $\Zp$ given by
\begin{eqnarray}
C = \Zp[x,x^{-1}] \otimes \Lambda_{\Zp}(e) \nonumber \\ \,\,\,\,\, |x|=2p-2, |e|=2p-3,\,\,\,\,\, d(x)=pe. \nonumber
\end{eqnarray}
This implies that
\[
d(x^m)=mpex^{m-1}=\epsilon p^{\nu(m)+1}ex^{m-1}
\]
where $\epsilon$ is a unit in $\Zp$.
Thus, the nontrivial homology groups of $C$ are given by
\[
H_{(2p-2)m-1}(C)=\mathbb{Z}/ p^{\nu(m)+1}\{ [ex^{m-1}]\} \cong \Zp/mp\Zp,\,\,\, m \in \mathbb{Z},\, m \neq 0
\]
and $\Zp$ in degrees $0$ and $-1$.

\begin{lemma}
Let $\gamma_m$, $m\neq0$ be the homology class of the element $-mex^{m-1}$, which has order $p$ in $H_{(2p-2)m-1}(C)$. Then the $\gamma_m$ satisfy the following Massey product relation.
\[
\left<\gamma_i,p,\gamma_j \right>= \gamma_{i+j}.
\]

Furthermore, the indeterminacy of this product is zero. 
\end{lemma}

\begin{proof}
This can be computed directly using the definition of Massey products given in \ref{subsec:toda}. In this case, we have
\[
a= -iex^{i-1} \,\,\, b=p, \,\,\, c=-jex^{j-1}
\]
and thus $u=-x^i$, $v=x^j$. When verifying this one has to take great care of using the correct signs. We see that the element
\[
x^i(-jex^{j-1})-iex^{i-1}x^j=-(i+j)ex^{i+j-1}
\]
lies in the Massey product. The indeterminacy of the product is
\[
\gamma_iH_{(2p-2)j}(C)\oplus\gamma_jH_{(2p-2)i}(C)
\]
which is zero because the homology groups in degrees that are nonzero multiples of $(2p-2)$ are already zero.
\end{proof}

\bigskip
We are going to show that this homology information determines $C$ up to quasi-isomorphism. 

The aim is to show that there is a quasi-isomorphism of dgas from $C$ to $D$, where $D$ is any other commutative dga with $H_*(C) \cong H_*(D)$ and Massey product relations
\[
\left<\alpha_i,p,\alpha_j \right>= \alpha_{i+j}
\]
where $\alpha_i$ is an order $p$ element of $H_{(2p-2)i-1}(D)$. 

The beauty of the test dga $C$ is that its underlying commutative algebra is free- hence one can construct a map to another dga by just specifying the map on the algebra generators and show that the resulting map of algebras is compatible with the differentials. 
The first intuitive step when constructing such a quasi-isomorphism is to send homology generators of $C$ to homology generators of $D$. 

We first show that there is a quasi-isomorphism of chain complexes (rather than dgas) from $C$ to $D$ as this will make the argument of the main theorem \ref{th:quasiiso} much easier to follow.

\begin{lemma}\label{lem:quasiisochain}
There is a quasi-isomorphism of chain complexes $\varphi: C \longrightarrow D$.
\end{lemma}

\begin{proof}
For this, we need to specify elements $$\varphi(ex^i) \in D_{(2p-2)(i+1)-1}\,\,\,\mbox{and}\,\,\, \varphi(x^i) \in D_{(2p-2)i}.$$ 
In other words, we have to find appropriate elements $a_i$ and $b_i$ in $D$ for which we can set
\[
\varphi(ex^{i-1}):= a_{i} \in D_{(2p-2)i-1}, 
\]
as well as sending $x^i$ to an element $$ \varphi(x^i):=b_i \in D_{(2p-2)i}.\,\,\, $$ To define a morphism of chain complexes, this choice has to be compatible with the differentials of $C$ and $D$.
We will now specify $a_i$ and $b_i$ for which the above has the desired properties.

For this to be a well-defined map of chain complexes we have to show that we can choose $a_i$ and $b_i$ to satisfy
\[
d(b_i)=ipa_i.
\]
First we look at the Massey product relation
\[
\left<\alpha_i,p,\alpha_i\right>=\alpha_{2i}.
\]
The element $\alpha_i$ has order $p$ in $H_{(2p-2)i-1}(D) = \mathbb{Z}/p^{\nu(i)+1}=\mathbb{Z}_p/ip\mathbb{Z}_p.$ So there must be an $a_i \in D_{(2p-2)i-1}$ with
\[
\alpha_i=[-i\cdot a_i].
\]
By definition of the Massey product there is an element ${b}_i \in H_{(2p-2)i}(D)$ such that $$d({b}_i)= ipa_i.$$ This $b_i $ is precisely the element we were looking for. 

We can conclude that
\[
\varphi: C \longrightarrow D
\]
defined via $\varphi(ex^{m-1})=a_m$ and $\varphi(x^m)=b_m$ with $a_m$ being a multiple of the homology generator and $b_m$ as above is a map of chain complexes. Obviously, it induces an isomorphism in homology as desired.
\end{proof}

We now continue by showing that the elements $a_m$ and $b_m$ can be chosen in such a way that $\varphi$ is multiplicative.

\begin{theorem}\label{th:quasiiso}
Let $D$ be a commutative dga whose homology and Massey products agree with those of $C$. Then there is a quasi-isomorphism of dgas $$\varphi: C \longrightarrow D.$$ 
\end{theorem}

\begin{proof}

We would like to show that the quasi-isomorphism $\varphi$ of Lemma \ref{lem:quasiisochain} is multiplicative. Explicitly, we have identified elements $a_i$ and $b_i$ in $D$ with
\[
\varphi(ex^i)=a_i, \,\,\,\,\,\varphi(x^i)=b_i \,\,\,\,\,\mbox{for}\,\,i \in \mathbb{Z}.
\]
For simplicity, we set 
\[
a:=a_1 \,\,\,\mbox{and}\,\,\, b:=b_1.
\]
To show that $\varphi$ is multiplicative, we have to show that we can pick the elements $a_i$ and $b_i$ in the proof of Lemma \ref{lem:quasiisochain} to be 
\[
a_i=ab^{i-1}\,\,\,\mbox{and}\,\,\,b_i=b^i \,\,\,\mbox{for}\,\,i \in \mathbb{Z}.
\]
We will do so in the following steps.
\begin{enumerate}
\item {\bf the nonnegative part:} show that $a_i=ab^{i-1}$ and $b_i=b^i$ for $i \ge 1$.
\item {\bf the nonpositive part:} identify an ``exterior'' generator $\bar{a} \in D_{-2p+1}$ and a ``polynomial'' generator $\bar{b}\in D_{-2p+2}$ and show that $b_{-i}=\bar{b}^i$ and $a_{-i}=\bar{a}\bar{b}^{i-1}$ for $i \ge 1$.
\item {\bf merging A:} As $D_0=\Zp$, we know that $d(b\bar{b})=\theta \in \Zp$. We use this to show that $\bar{a}b=a\bar{b}.$
\item {\bf merging B}: use a Massey product relation to show that this number $\theta$ is a unit and hence $\bar{b}=b^{-1}$.
\end{enumerate}

{\bf The nonnegative part:} We are going to prove inductively that we can pick the elements $a_i$ and $b_i$ to be 
\[
a_i=ab^{i-1} \,\,\,\mbox{and}\,\,\, b_i=b^i \,\,\,\mbox{for}\,\, i \ge 1.
\]
For this we need to show
\begin{itemize}
\item The element $ab^{i-1}$ is a cycle. Its homology class is a generator for the cyclic group $H_{(2p-2)i-1}(D)=\mathbb{Z}/p^{\nu(i)+1}.$
\item The differential on $b^i$ is $d(b^i)=ipab^{i-1}$.
\end{itemize}
The start of our induction is easy as this is exactly how we picked $a_1$ and $b_1$ in Lemma \ref{lem:quasiisochain}: $a=a_1$ gives a generator of the cyclic group $H_{2p-3}(D)=\mathbb{Z}/p$ and $d(b)=d(b_1)=pa_1=pa.$ This also shows the second bullet point that $d(b^i)=piab^{i-1}.$

Again, by $\alpha_i$ we denote the order $p$ element 
\[
\alpha_i=[-ia_i] \in H_{(2p-2)i-1}(D)\cong\mathbb{Z}/p^{\nu(i)+1}.
\]
Now let us assume that we can choose
\[
a_{i-1}=ab^{i-2} \,\,\,\mbox{and}\,\,\, b_{i-1}=b^{i-1} \,\,\,\mbox{for}\,\,i \ge 2.
\]
The given Massey product relations tell us that
\[
\alpha_i=\left<\alpha_1, p, \alpha_{i-1}\right>
\]
which by the definition of Massey products using the explicit representatives equals
\[
\alpha_i=[-(i-1)a_{i-1}b-ab_{i-1}].
\]
Using our induction assumption we have
\[
\alpha_i=[-(i-1)ab^{i-2}b-ab^{i-1}]=[-iab^{i-1}].
\]
This means that $[-iab^{i-1}]$ is an element of order $p$ in $H_{(2p-2)i-1}(D)$ and that consequently, $[ab^{i-1}]$ is a generator of this homology group, which is what we wanted to show.

\bigskip
{\bf The nonpositive part:} This is very similar to the proof concerning the nonnegative part. Using the notation of Lemma \ref{lem:quasiisochain} we set
\[
\bar{a}:=a_{-1}\,\,\,\mbox{and}\,\,\, \bar{b}=b_{-1},
\]
i.e. $\bar{a}$ is the cycle generating $H_{-2p+1}(D)=\mathbb{Z}/p$ and $\bar{b}$ is a nonzero element in $D_{-2p+2}$ with $d(\bar{b})=-p\bar{a}.$ We will show that we can pick
\[
a_{-i}=\bar{a}\bar{b}^{i-1}\,\,\,\mbox{and}\,\,\,b_{-i}=\bar{b}^i\,\,\,\mbox{for}\,\, i \ge 1.
\]
Assume that we already know that 
\[
a_{-(i-1)}=\bar{a}\bar{b}^{i-2}\,\,\,\mbox{and}\,\,\,b_{-(i-1)}=\bar{b}^{i-1}\,\,\,\mbox{for}\,\, i \ge 2.
\]
By assumption we have the Massey product relation
\[
\alpha_{-i}=\left<\alpha_{-1},p,\alpha_{-(i-1)}\right>=[ia_{-i}].
\]
By definition this equals
\[
\alpha_{-i}=[(i-1)\bar{a}\bar{b}^{i-2}\bar{b}+\bar{a}\bar{b}^{i-1}]=[i\bar{a}\bar{b}^{i-1}].
\]
Analogously to the previous argument this means that $\bar{a}\bar{b}^{i-1}$ is a generator of $H_{-(2p-2)i-1}(D)$ and thus we can pick 
\[
a_{-i}=\bar{a}\bar{b}^{i-1}\,\,\,\mbox{and}\,\,\,b_{-i}=\bar{b}^i\,\,\,\mbox{for}\,\, i \ge 1
\]
as desired.

\bigskip
{\bf Merging A:}
The Postnikov argument of Proposition \ref{prop:degreezero} showed that without loss of generality we have $D_0=\Zp$ generated by the unit of $D$. Thus, multiplying the positive and negative polynomial generators of the previous two steps gives
\[
b \bar{b}=\theta \in \Zp.
\]
We use this to get some useful relations between the positive and negative parts of $D$. First of all, this gives us
\[
d(b \bar{b})=0
\]
but also 
\[
d(b\bar{b})=d(b)\bar{b}+bd(\bar{b})=pa\bar{b}-p\bar{a}b,
\]
so
\begin{equation}\label{eq:minusone}
pa\bar{b}=p\bar{a}b \in D_{-1}. 
\end{equation}
The elements $a\bar{b}$ and $\bar{a}b$ are cycles in $D_{-1}$. As the differential $$d:D_0=\Zp \longrightarrow D_{-1}$$ is trivial, the group of cycles in $D_{-1}$ is already $H_{-1}(D)=\Zp$ which has no zero-divisors. Thus, it follows from (\ref{eq:minusone}) that
\[
a\bar{b}=\bar{a}b,
\]
which will be a key ingredient to our next step.

\bigskip
{\bf Merging B:} To show that our quasi-isomorphism $\varphi:C \longrightarrow D$ with $\varphi(x)=b_1=b$ and $\varphi(x^{-1})=b_{-1}=\bar{b}$ from Lemma \ref{lem:quasiisochain} is multiplicative, we have to show that $b\bar{b}=1$. We know that $b\bar{b}=\theta$ for some $\theta \in \Zp$. It suffices to show that $\theta$ is a unit, i.e. not divisible by $p$. We do so by exploiting the given Massey product relation
\[
\alpha_1=\left<\alpha_2, p, \alpha_{-1}\right>.
\]
We know that $$\alpha_2=[-2ab], \alpha_{-1}=[\bar{a}]\,\,\,\mbox{and}\,\,\, \alpha_1=[-a].$$ Putting this into the defining equation of the Massey product yields
\begin{multline}
\left<\alpha_2, p, \alpha_{-1}\right>= [-2ab\bar{b}+\bar{a}b^2] = [-2a\theta+\bar{a}b^2] \nonumber\\=[-2a\theta + a \bar{b}b]=[-2a\theta + a\theta]=\theta[-a].\nonumber
\end{multline}
The second and fourth equality use that $b \bar{b}=\theta$, and the third equality uses the previously proved relation $\bar{a}b=a\bar{b}$. But we also know that 
\[
\left<\alpha_2, p, \alpha_{-1}\right>=[-a]
\]
by assumption, hence $\theta[-a]=[-a].$ As $[-a]$ is a nonzero element of order $p$, this is only possible if $\theta \in \Zp$ is a unit, which is what we wanted to show.

\bigskip
To summarise, we now know that
\[
a_i=ab^{i-1}, \,\,\, b_i=b^{i}, \,\,\, a_{-i}=\bar{a}\bar{b}^{i-1}, \,\,\,  b_{-i}=\bar{b}^{i}\,\, (\mbox{all for $i \ge 1$}), \,\,\, \bar{b}=b^{-1}\,\,\,\mbox{and}\,\,\,\bar{a}=ab^{-2}
\]
and thus that $$a_i=ab^{i-1}\,\,\,\mbox{and}\,\,\, b_i=b^i \,\,\,\mbox{for}\,\,\, i \in \mathbb{Z}.$$ This concludes the proof that the map of chain complexes given by $\varphi(e)=a$ and $\varphi(x)=b$ is a quasi-isomorphism of dgas.

\end{proof}

\section{Conclusion and directions}\label{sec:conclusions}

\subsection{Uniqueness conclusions}

In Theorem \ref{th:quasiiso} we showed that any commutative dga whose homology and Massey products agree with the homotopy groups and Toda brackets of $\LS$ is quasi-isomorphic to the dga $C$ given by  
\begin{eqnarray}
C = \Zp[x,x^{-1}] \otimes \Lambda_{\Zp}(e) \nonumber \\ \,\,\,\,\, |x|=2p-2, |e|=2p-3,\,\,\,\,\, d(x)=pe. \nonumber
\end{eqnarray}
Together with Theorem \ref{thm:premonoid} we can conclude the following.

\begin{corollary}
There is at most one pre-monoidal algebraic model category for $\Ho((\LS)\mbox{-mod})$. 
\qed
\end{corollary}

\begin{corollary}
Any algebraic model category whose homotopy category is pre-monoidally equivalent to $\Ho((\LS)\mbox{-mod})$ is Quillen equivalent to the category of dg-modules over $C$ with $C$ given above.
\qed
\end{corollary}

Unfortunately, we do not know if there actually exists a pre-monoidal equivalence
\[
\Phi: \Ho((\LS)\mbox{-mod}) \longrightarrow \Ho(C\mbox{-mod}).
\]
If the answer to that question is no, then we can conclude that there is simply no pre-monoidal algebraic model of $\Ho((\LS)\mbox{-mod})$ at all. 

\subsection{The commutativity issue}

Can we actually do without the commutativity assumption on our dga? If yes, then we can prove that there is at most one algebraic model for $\Ho((\LS)\mbox{-mod})$ without the pre-monoidality assumption, recall the discussion in Section \ref{sec:commutative}. For that, a possible strategy would be to find an example of a non-commutative dga with the right homology and Massey products and show that it is unique up to quasi-isomorphism analogously to Theorem \ref{th:quasiiso}. However, so far we have not been able to produce a non-commutative dga with the correct data- it is not clear if any non-commutative examples exist which would not be quasi-isomorphic to our (commutative) test dga $C$.

\subsection{Further chromatic steps}

The category $L_{K(1)}\mathbb{S}$-modules is \emph{not} Quillen equivalent to $K(1)$-local spectra $L_{K(1)}\Sp=L_{K(1)}(\mathbb{S}\mbox{-mod})$. For $E(1)$-local spectra the situation is different as $E(1)$-localisation is smashing- localising with respect to $E(1)$ is the same as localising with respect to the $E(1)$-local sphere. Thus, the model category of $E(1)$-local spectra is Quillen equivalent to the category of $L_{E(1)}\mathbb{S}$-modules. However, this does not hold true for $K(1)$. The $K(1)$-local stable homotopy category does possess a compact generator, even though it is not the sphere: the mod-$p$ Moore spectrum $L_{K(1)}M$ \cite[Section 3.2]{SchShi03}. So a way of showing uniqueness of algebraic models for $\Ho(L_{K(1)}\Sp)$ would be showing that there is only one dga modelling the endomorphisms and Toda brackets of $L_{K(1)}M$. However, it is very hard to get a grasp on the necessary computations. For example, the homology of such a dga would not even be commutative as the element $v_1$ in $[M, M]^{L_{K(1)}\Sp}_{2p-2}$ is not central \cite{CraKna88}.

\bigskip

It would be interesting to see if the computations from this paper using the data of $\LS$ could be used to prove an analogous result for the $E(1)$-local sphere instead. That would lead to pre-monoidal uniqueness of algebraic models for $\Ho((L_1\mathbb{S})\mbox{-mod})$ and hence for $\Ho(L_1\Sp)$. An algebraic model for $\Ho(L_1\Sp)$ exists in Franke's model \cite{Fra96}, see also \cite{Roi08}, which by \cite{Gan07} and \cite{BarRoi11a} is even pre-monoidally equivalent to $\Ho(L_1\Sp)$. If we could extend Theorem \ref{th:quasiiso} to $L_1\mathbb{S}$, then this would show that Franke's algebraic model is indeed the only algebraic model for $\Ho(L_1\Sp)$ that retains some monoidal information.

The homotopy groups of $L_1\mathbb{S}$ are not that different from $\pi_*(\LS)$- they only differ by rational parts in degrees -1 and -2.
Those two spheres are closely related by the homotopy pullback square
\[
\xymatrix{ L_1 \mathbb{S} \ar[d]\ar[r]& L_{K(1)} \mathbb{S}\ar[d] \\
L_\mathbb{Q} \mathbb{S} \ar[r] & L_\mathbb{Q} L_{K(1)}\mathbb{S} 
}
\]
\cite[3.9]{Dwy04}. Unfortunately, performing an analogous pullback of dgas with our test dga $C$ does not even give a test dga for the $E(1)$-local case: the resulting dga $C'$ has the wrong multiplicative structure. The multiplication
\[
\pi_{(2p-2)i-1}(L_1\mathbb{S})\otimes \pi_{-(2p-2)i-1}(L_1\mathbb{S}) \longrightarrow \pi_{-2}(L_1\mathbb{S})
\]
is injective \cite[Theorem 8.2.10(d)]{Rav86} whereas the multiplication in the homology of the pullback
\[
H_{(2p-2)i-1}(C') \otimes H_{-(2p-2)i-1}(C') \longrightarrow H_{-2}(C') 
\]
is not. This is to be expected- performing any sort of homological localisation on an algebraic model category will almost always lead to trivial results \cite{BarRoi13b}. Hence the homotopical behaviour of algebraic model categories in this context is going to remain a subject of future research.

\bibliographystyle{plain}

\end{document}